\documentclass{amsart}
\usepackage{comment}

\title{Representation of Unity by Binary Forms}
\author{Shabnam Akhtari}
\address{Max-Planck-Institut f\"{u}r Mathematik, 
 Vivatsgasse 7,
53111 Bonn,
Germany\newline}
 \email {akhtari@mpim-bonn.mpg.de}
 

\begin{document}

\newtheorem{thm}{Theorem}[section]
\newtheorem{prop}[thm]{Proposition}
\newtheorem{lemma}[thm]{Lemma}
\newtheorem{cor}[thm]{Corollary}
\newtheorem{conj}[thm]{Conjecture}
\begin{abstract}
In this paper, it is shown that if $F(x , y)$ is an irreducible binary form with integral coefficients and degree $n \geq 3$, then provided that the absolute value of the discriminant of $F$ is large enough,  the equation $F(x , y) = \pm 1$ has at most $11n-2$ solutions in integers $x$ and $y$. We will also establish some sharper bounds when more restrictions are assumed. These upper bounds are derived by combining methods from classical analysis and geometry of numbers. The theory of linear forms in logarithms plays an essential role in studying the geometry of our Diophantine equations.

\end{abstract}
\maketitle

\section{Introduction}\label{Intro}

Let $F(x , y) = a_{n}x^n + a_{n-1}x^{n-1}y + \ldots + a_{0} y^n$ be an irreducible binary form with rational integer coefficients and $n \geq 3$. We will study $N_{n}$, the number of solutions to the equation 
\begin{equation}\label{1.2}
F(x , y) = \pm 1,
\end{equation}
   in integers $x$ and $y$.  We will regard $(x , y)$ and $(-x , -y)$ as one solution. So we may only count the solutions with $y\geq 0$. But how large can $N_{n}$ be? Let $p$ be a prime and consider the following irreducible form
$$
F_{1}(x , y) = x^n + p (x - y)(2x - y)\ldots (nx - y).
$$
It is easy to see that $F_{1}(x , y)= 1$ has the following $n$ solutions 
$$(1 , 1),(1 , 2), \ldots , (1 , n).$$
  Thus a linear upper bound of the shape $cn$ is best possible except for the determination of $c$. We will show that
\begin{thm}\label{main}
Let  $F(x , y)$ be an irreducible binary form with integral coefficients and degree $n \geq 3$. Then the Diophantine equation $|F(x , y)| = 1$ has at most $11n-2$ solutions in integers $x$ and $y $, provided that the absolute value of the discriminant of $F$ is greater than $D_{0}$, where $D_{0} = D_{0}(n)$ is an effectively computable constant. 
Moreover, assume that the polynomial $F(x , 1)$ has $r$ real roots and $2s$ non-real roots ($r+2s = n$). Then $|F(x , y)| = 1$ has at most $11r +4s -1$ solutions in integers $x$ and $y$.
\end{thm}
We remark here that $D_{0}$ can be computed effectively in terms of $n$, the degree of $F$. Indeed, we may take $ D_{0} = 2^{22} (n+1)^{10} n^n $.  Theorem \ref{mat} gives an algorithm to compute $D_{0}$.

In the above theorem, we supposed that $F$ is irreducible. We will see in Section \ref{RF}, that when $F$ is reducible, the situation is simpler. 
 Let $D$ be the discriminant of form $F$ (it is defined in Section \ref{EF}). Note that the condition $|D|> D_{0}(n)$ is a restriction, because we know for binary form $F \in \mathbb{Z}[x , y]$ of degree $n$ and discriminant $D \neq 0$, we have the following sharp bound (see \cite{Gyo23}):
 \begin{equation}\label{nDG}
 n \leq 3 + 2\left( \log |D|\right) / \log 3 .
 \end{equation}

In Section \ref{EF}, we will see how Theorem \ref{main} gives  an upper bound for the number of integral solutions to $F(x , y) = \pm 1$ when $F$ has a small discriminant.

One may conjecture that the number of solutions may be estimated in terms of $r$ the number of real solutions of $F(x , 1) = 0$. This is not the case. For example, let $n$ be an even integer and $p$ a prime number. If we put
$$
F(x , y) = x^{n} + p (x - y)^2 (2x -y)^2 \ldots (\frac{n}{2} x - y)^2
$$
then $F(x , y)$ is irreducible and $F(x , 1) = 0$ has no real root. However, $F(x , y) = 1$ has the following solutions:
$$
(1 , 1), (1, 2), \ldots, (1, \frac{n}{2}).
$$ 
In Proposition \ref{Grp}, we will show that the number of solutions $(x , y)$ with large enough  $y$ can be estimated in terms of $r$.

 In 1909, Thue \cite{Thu} derived the first general sharpening of Liouville's theorem on rational approximation to algebraic numbers, proving, if $\theta$ is algebraic of degree $n \geq 3$ and $\epsilon > 0$, that there exists a constant $c(\theta , \epsilon)$ such that
$$
\left| \theta - \frac{p}{q} \right| > \frac{c(\theta , \epsilon)}{q^{\frac{n}{2} + 1 + \epsilon}}
$$
for all $p \in \mathbb{Z}$ and $q \in \mathbb{N}$. It follows almost immediately, if $F(x , y)$ is an irreducible binary form in $\mathbb{Z}[x , y]$ of degree at least three and $h$ a nonzero integer, that the equation 
\begin{equation}\label{1.2h}
F(x , y) = h
\end{equation}
has only finitely many solutions in integers $x$ and $y$. Equation (\ref{1.2h}) is called a {\emph Thue equation}. 

For any nonzero integer $h$ let $\omega(h)$ denote the number of distinct prime factors of $h$. In 1933, Mahler \cite{Mah23} proved that equation (\ref{1.2h}) has at most $C_{1}^{1 + \omega(h)}$ solutions in co-prime integers $x$ and $y$, where $C_{1}$ is a positive number that depends on $F$ only.
In 1987, Bombieri and Schmidt \cite{Bom} showed that the number of solutions of $F(x , y) = h$ in co-prime integers $x$ and $y$ is at most 
$$
C_{2} \,  n^{1 + \omega(h)},
$$
where $C_{2}$ is an absolute constant. Further they showed that $C_{2}$ may be taken $215$ if $n$ is sufficiently large. Note that this upper bound is independent of the coefficients of the form $F$; a result of this flavour was first deduced in 1983 by Evertse \cite{Eve}.
In the introduction of \cite{Bom}, Bombieri and Schmidt comment that their argument can be used to prove a more general result. For example,  if $N_{n}$ is the corresponding bound in the special case $h = 1$, one obtains $N_{n}n^{\omega(h)}$ as a bound in the general case. For this reason we will focus on the equation $|F(x , y)| = 1$.

The effective solution of an arbitrary Thue equation has its origin in Baker's \cite{Bak} theorem that says that if $\kappa > n+1$, then every integer solution $(x , y)$ of equation (\ref{1.2h}) satisfies
$$
\max \{ |x|, |y|\} < C_{3} \exp{\log ^{\kappa} |h|}
$$
where $C_{3}$ is an effectively computable constant depending only on $n$, $\kappa$ and the coefficients of $F$.

Evertse and Gy\H{o}ry (see \cite{Eve} and \cite{EG16}) have studied the Thue inequality
\begin{equation}\label{1.2ih}
0 < |F(x , y)| \leq h.
\end{equation} 
Define, for $ 3 \leq n < 400$
$$
\left( N(n) , \delta(n) \right) = 
\left( 6 n 7^{\binom{n}{3}},\,  \frac{5}{6} n (n-1) \right) $$
and for $n > 400$
$$
\left( N(n) , \delta(n) \right) = \left( 6n , \, 120(n -1)\right). 
$$
They prove that if
$$
|D| > h^{\delta{n}} \exp (80 n (n -1)),
$$
then the number of solutions to (\ref{1.2ih}) in co-prime integers $x$ and $y$ is at most $N(n)$.

Gy\H{o}ry \cite{Gyo1} also shows, for binary form $F$ of degree $n \geq 3$, that if $0 <a < 1$ and 
$$
\left| D \right| \geq n^n (3.5^{n} h^2)^{\left(2(n-1)/(1-a)\right)},
$$
then
the number of solutions to (\ref{1.2ih}) in co-prime integers $x$ and $y$ is
at most $25n + (n+2) \left(\frac{2}{a} + \frac{1}{4}\right)$, and if $F$ is reducible then at most $5n + (n+2) \left(\frac{2}{a} + \frac{1}{4}\right)$.

A great reference in this field is a work of  Stewart \cite{Ste}. We will  follow many arguments from \cite{Ste} here. A consequence 
 of  Stewart's main theorem in \cite{Ste} is that if the discriminant $D$ of $F$ is non-zero and 
$$
|D|^{1/n(n-1)} \geq |h|^{\frac{2}{n+\epsilon}},
$$
then the number of pairs of co-prime integers $(x , y)$ for which $F(x , y) = h$ holds is at most 
$$
1400 \left( 1 + \frac{1}{8 \epsilon n}\right) n.
$$

Bennett \cite{Ben} and Okazaki \cite{Oka} have obtained very good upper bounds for the number of solutions to cubic Thue equations. Some upper bounds are given for the number of integral solutions to quartic Thue equations in \cite{Akh0} and \cite{AO1}. Throughout this paper we may assume $n$, the degree of our binary form, is greater than $4$. 

We will use methods from \cite{Ste} to give upper bounds on the number of ``small'' solutions to (\ref{1.2}). Then, in Section \ref{LC}, we will generalize some ideas from \cite{Oka, AO1} to associate a transcendental curve $\phi(x , y)$ to the binary form $F(x , y)$. Introducing this curve will give us the opportunity to bring the theory of linear forms in logarithms in.

\section{Reducible Forms}\label{RF}
Let us take a brief interlude from the principal matter at hand to discuss the much simpler situation where the form $F(x , y)$ is reducible over $\mathbb{Z}[x, y]$. In general, equation (\ref{1.2}) may have infinitely many integral solutions; $F(x , y)$ could, for instance, be a power of a linear or indefinite binary quadratic form that represents unity. If $F(x, y)$ is a reducible form, however, we may very easily derive a stronger version of our main theorem under the assumption that $F(x , 1)$ has at least two distinct zeros.

Suppose that $F(x , y)$ is reducible and can be factored over $\mathbb{Z}[x , y]$ as follows
$$ 
F(x , y) = F_{1}(x , y) F_{2}(x , y),
$$
with $\textrm{deg}(F_{1}) \leq \textrm{deg}(F_{2})$ and $F_{1}$ irreducible over $\mathbb{Z}[x , y]$. Therefore, the following equations must be satisfied:
\begin{equation}\label{e1}
F_{1}(x , y) = \pm 1
\end{equation}
and  
\begin{equation}\label{e2}
F_{2}(x , y) = \pm 1.
\end{equation}
This means the number of solutions to (\ref{1.2}) is no more than the minimum of number of solutions to (\ref{e1}) and (\ref{e2}). 

First suppose that $F_{1}$ is a linear form. Then the equation (\ref{e2}) can be written as a polynomial of degree at most $n-1$ in $x$ and therefore there are no more than $2(n-1)$ complex solutions to above equations. 
 
Now let us suppose that $F_{1}$ is a quadratic form. Using B\'ezout's theorem from classical  algebraic geometry concerning the number of common points of two plane algebraic curves, we conclude that (\ref{1.2}) has at most $4 (n-2)$ integral solutions.


If $\textrm{deg}(F_{1}) \geq 3$ then Theorem \ref{main} will give us an upper bound for the number of integral solutions to (\ref{e1}), and therefore to (\ref{1.2}).

\section{Equivalent Forms}\label{EF}
 Our approach depends on the fact that if we transform $F$ by the action of an element of $GL_{2}(\mathbb{Z})$ the problem of counting solutions remains unchanged, while the Diophantine approximation properties of $F$ can change very drastically.
Let 
$$
A = \left( \begin{array}{cc}
a & b \\
c & d \end{array} \right)
$$
and define the binary form $F_{A}$ by
$$
F_{A}(x , y) = F(ax + by \ ,\  cx + dy).$$
If the determinant of matrix $A$ is equal to $\pm 1$ then we say that $F_{A}$ and $-F_{A}$ are equivalent to $F$.

Suppose that  $A \in GL_{2}(\mathbb{Z})$ and $(x , y)$ is a solution of (\ref{1.2}) in integers $x$ and $y$. Then 
$$A \left( \begin{array}{c} 
x\\
y
\end{array} \right)= \left( \begin{array}{c} 
ax+by\\
cx+dy
\end{array} \right)$$
and $(ax+by , cx+dy)$ is a solution of $F_{A^{-1}}(x , y) = \pm 1$ in integers $x$ and $y$.

Let $F$ be a binary form that factors in $\mathbb{C}$ as
$$
\prod_{i=1}^{n} (\alpha_{i} x - \beta_{i}y).
$$
The discriminant $D_{F}$ of $F$ is given by
$$
D_{F} = \prod_{i<j} (\alpha_{i} \beta_{j} - \alpha_{j}\beta_{i})^2.
$$
Observe that for any $2 \times 2$ matrix $A$ with integer entries
\begin{equation}\label{St6}
D_{F_{A}} = (\textrm{det} A)^{n (n-1)} D_{F}.
\end{equation}

 We denote by $N_{F}$ the number of solutions in integers $x$ and $y$ of the Diophantine equation (\ref{1.2}). If $F_{1}$ and $F_{2}$ are equivalent  then $N_{F_{1}} = N_{F_{2}}$  and  $D_{F_{1}} = D_{F_{2}}$. 
 
Let $p$ be a prime number and put 
$$
A_{0} = \left( \begin{array}{cc}
p & 0 \\
0 & 1 \end{array} \right)\, , \qquad 
A_{j} = \left( \begin{array}{cc}
0 & -1 \\
p & j \end{array} \right)
$$
for $j = 1, \ldots, p$. Then we have
$$
\mathbb{Z}^2 = \cup_{j=0}^{p} A_{j} \mathbb{Z}^2.
$$
 Therefore the number of solutions of (\ref{1.2}) is at most $N_{F_{0}} + N_{F_{1}} + \ldots + N_{F_{p}}$, where
$$
F_{j} (x , y) = F_{A_{j}}(x , y).
$$
Note that by (\ref{St6}), 
$$
\left| D_{F_{A_{j}}}\right| \geq p^{n(n-1)}.
$$
 Therefore, if $N$ is an upper bound for the number of solutions to (\ref{1.2}) for binary forms $F$ with $|D_{F}| \geq p^{n(n-1)}$ then  $(p + 1)N$ will be an upper bound for the number of solutions to $|F(x , y)| = 1$ when $F$ has a nonzero discriminant. 

Assume  that $F(x , y) = \pm 1$ has a solution $(x_{0} , y_{0})$. Then there is a matrix $A$ in $GL_{2}(\mathbb{Z})$ for which 
$A^{-1} (x _{0} , y_{0})$ is $(1 , 0)$. Therefore, $(1 , 0)$ is a solution to $$F_{A}( x , y) = \pm 1.$$
 We conclude that either $F_{A}$ or $-F_{A}$ is a monic form. From now on we will assume that the binary form $F(x , y)$ in Theorem \ref{main} is monic.

\section{Heights}\label{Hei}

 In this section we give a brief review of the theory of height functions of polynomials and binary forms.

For the polynomial $G(x) = c (x - \beta_{1})\ldots  (x - \beta_{n})$ with $c \neq 0$, the Mahler measure $M(G)$ is defined by
$$
M(G) = |c| \prod_{i =1}^{n} \max(1 , \left| \beta_{i}\right|).
$$
 Mahler \cite{Mah} showed, for polynomial $G$ of degree $n$ and discriminant $D$, that
\begin{equation}\label{mahD5}
 M(G) \geq \left(\frac{D}{n^n}\right)^{\frac{1}{2n -2}}.
 \end{equation}
The Mahler measure of an algebraic number $\alpha$ is defined as the Mahler measure of the minimal polynomial of $\alpha$ over Q.


For an algebraic number $\alpha$,  the (naive) height of $\alpha$, denoted by $H(\alpha)$, is defined by the following identities.
$$
H(\alpha) =  H\left(f(x)\right) =\max  \left( |a_{n}|, |a_{n-1}|, \ldots , |a_{0}|\right)
$$
where $f(x) = a_{n}x^{n} + \ldots + a_{1}x + a_{0}$ is the minimal polynomial of $\alpha$ over $\mathbb{Z}$.

We have 
\begin{equation}\label{Lan5}
  {n \choose \lfloor n/2\rfloor}^{-1} H(\alpha) \leq M(\alpha) \leq (n+1)^{1/2} H(\alpha).
\end{equation}
We will use transformations in $GL_{2}(\mathbb{Z})$ to dispense with a technical hypothesis about the height of $F$.
We call  the polynomials $f(x)$ and $f^{*}(x) \in \mathbb{Z}$ strongly equivalent if $f^{*}(x) = f(x + a)$ for some $a \in \mathbb{Z}$.
 Two algebraic integers $\alpha$ and $\alpha'$ are called strongly equivalent if their minimal polynomials are strongly equivalent.

 \begin{prop}\label{gyor5}(Gy\H{o}ry \cite{Gyo23})
 Suppose that $f(x)$ is a monic polynomial in $\mathbb{Z}[x]$ with degree $n \geq 2$ and non-zero discriminant $D$. There is a
 polynomial $f^{*}(x) \in \mathbb{Z}$ strongly equivalent to $f(x)$ so that 
 $$
 H\left(f^{*}(x)\right) < \exp\{ n^{4n^{12}} |D|^{6n^8}\} < \exp \exp \{4\left( \log |D|\right)^{13}\}.
$$
 \end{prop}

For polynomial  $f(x) = a_{n}x^{n} + \ldots + a_{1}x + a_{0}$ with degree $n$ and integer coefficients, put 
$$
L(f) = \left|a_{n}\right|+ \ldots + \left|a_{1}\right| + \left|a_{0}\right|.
$$ 
Mahler \cite{Mah3} showed that
\begin{equation}\label{ML}
2^{-n}L(f) \leq M(f) \leq L(f).
\end{equation}

Define  the absolute logarithmic height of an algebraic number as follows. Let $\alpha_{1}$ be a root of $F(x , 1) = 0$ and $\mathbb{Q}(\alpha)^{\sigma}$  the embeddings of  $\mathbb{Q}(\alpha)$ in $\mathbb{C}$,  $1 \leq \sigma \leq n$.  For   $ \rho \in  \mathbb{Q}(\alpha)$,  we respectively have $n$ Archimedean valuations of $\mathbb{Q}(\alpha)$:
$$
|\rho|_{\sigma} = \left|\rho^{(\sigma)}\right| , \   \   1 \leq \sigma \leq n.
$$
We enumerate simple ideals of $\mathbb{Q}(\alpha)$ by indices $\sigma > n$ and define
 non-Archimedean valuations of $\mathbb{Q}(\alpha)$ by the formulas
  $$
   |\rho | _{\sigma} = (  \textrm{Norm} \  \mathfrak{p})^{-k}, 
   $$ 
 where
 $$\
  k = \textrm{ord}_ {\mathfrak{p}} (\alpha) , \   \mathfrak{p} = \mathfrak{p}_{\sigma} , \   \sigma > n ,
  $$
 for any $\rho \in \mathbb{Q}(\alpha)^{*}$.  Then we have the \emph{product formula} :
 $$
  \prod_{1}^{\infty} |\rho|_{\sigma} = 1 ,  \   \rho \in  \mathbb{Q}(\alpha)^{*} .
  $$
  Note that $|\rho|_{\sigma} \neq 1$ for only finitely many $\rho$. We should also remark that if $\sigma_{2} = \bar{\sigma}_{1}$, i.e., 
$$\sigma_{2}(x) = \bar{\sigma}_{1}(x) \qquad \textrm{for} \qquad  x \in \mathbb{Q}(\alpha),
$$
then the valuations $|\, .\, |_{\sigma_{1}}$ and $|\, .\, |_{\sigma_{2}}$ are identical.
 We define the \emph{absolute logarithmic height} of $\alpha$ as
 $$
 h(\alpha) = \frac{1}{2n} \sum _{\sigma = 1}^{\infty}\left|\log |\alpha|_{\sigma}\right| . 
 $$
 This height is called absolute because it is independent of the field in which the number  $\alpha$ lies.

The following Lemmata about the height of algebraic numbers will be helpful later.

\begin{lemma}\label{p31}
For every non-zero algebraic number $\alpha$, we have $h(\alpha^{-1}) = h(\alpha)$. For algebraic numbers $\alpha_{1}, \ldots, \alpha_{n}$, we have 
\begin{equation*}\label{bj15}
h(\alpha_{1} \ldots \alpha_{n}) \leq h(\alpha_{1})+ \ldots + h(\alpha_{n})
\end{equation*}
and
\begin{equation*}\label{bj25}
h(\alpha_{1} + \ldots + \alpha_{n}) \leq \log n + h(\alpha_{1})+ \ldots + h(\alpha_{n}).
\end{equation*}
\end{lemma}
\begin{proof}
See \cite{Coh} for a proof.
\end{proof}

\begin{lemma}\label{hM5}(Voutier \cite{Vou})
 Suppose $\alpha$ is a non-zero algebraic number of degree $n$ which is not a root of unity. If $n \geq 2$ then
 \begin{equation*}
 h(\alpha) = \frac{1}{n} \log M(\alpha)>\frac{1}{4n} \left(\frac{\log \log n}{\log n}\right)^3.
 \end{equation*}
 \end{lemma}

 \begin{lemma}\label{mahl5}(Mahler \cite{Mah})
 If $a$ and $b$ are distinct zeros of polynomial $P(x)$ with degree $n$, then we have
\begin{equation*}
  | a - b | \geq \sqrt{3} (n+1)^{-n} M(P)^{-n+1},
  \end{equation*}
 where $M(P)$ is the Mahler measure of $P$.
  \end{lemma}
   
 In the following lemma  we approximate the size of $f'(\alpha)$ in terms of the discriminant and heights of $f$ , where $f'$ is the derivative of the polynomial $f$ and $\alpha$ is a root of $f = 0$.

  \begin{lemma}\label{esug}
 Let $f(x) = a_{n}x^{n} + \ldots + a_{1}x + a_{0}$ be an irreducible  polynomial of degree $n$ and with integral coefficients.  Suppose that  $\alpha_{m}$ is a root of $f(x) = 0$. For $f'(x)$ the derivative of $f$, we have
$$
2^{-(n-1)^2} \frac{\left|D_{f}\right|}{M(f)^{2n -2}}\leq    |f'(\alpha_{m})| \leq \frac{n(n+1)}{2} H(f) \left(\max (1 , |\alpha_{m}|)\right)^{n-1},
$$
where $D_{f}$ is the discriminant, $M(f)$ is the Mahler measure and $H(f)$ is the naive height of $f$.
 \end{lemma}
 \begin{proof}
The right hand side inequality is trivial by noticing that
$$
f'(x) = n a_{n}x^{n-1} + \ldots + a_{1}x.
$$
To see the left hand side inequality,  observe that for $\alpha_{i}$, $\alpha_{j}$, two distinct roots of $f(x)$, we have
 $$
 |\alpha_{i} - \alpha_{j}| \leq 2 \max(1 , |\alpha_{i}|) \max (1 , |\alpha_{j}|). 
 $$
Then
 \begin{eqnarray*}
 |f'(\alpha_{m})| & = & \prod_{i=1, i \neq m}^{n} |\alpha_{i} - \alpha_{m}| \geq \prod_{i=1, i \neq m}^{n} \frac{|\alpha_{i} - \alpha_{m}|}{\max(1 , |\alpha_{i}|) \max (1 ,
 |\alpha_{m}|)}\\
 & \geq & 2^{n-1 - n(n-1)} \prod _{j=1}^{n}\prod_{i=1, i \neq j}^{n} \frac{|\alpha_{i} - \alpha_{j}|}{\max(1 , |\alpha_{i}|) \max (1 ,
 |\alpha_{j}|)}\\
 & = & 2^{-(n-1)^2} \frac{\left|D_{F}\right|}{M(F)^{2n -2}}.
 \end{eqnarray*}
 \end{proof}

Suppose that $\mathbb{K}$ is an algebraic number field of degree $d$ over $\mathbb{Q}$ embedded in $\mathbb{C}$. If $\mathbb{K} \subset \mathbb{R}$, we put $\chi = 1$, and otherwise $\chi = 2$. We are given numbers $\gamma_{1} , \ldots, \gamma_{n} \in \mathbb{K}^{*}$ with absolute logarithmic heights $h(\gamma_{j})$, $1\leq j \leq n$. Let $\log \gamma_{1}$ , $\ldots$ , $\log \gamma_{n}$ be arbitrary fixed non-zero values of the logarithms. Suppose that 
$$
A_{j} \geq \max \{dh(\gamma_{j}) , |\log \gamma_{j}| \}, \  \   1 \leq j \leq n.
$$
Now consider the linear form 
$$
\mathfrak{L} = b_{1}\log\gamma_{1} + \ldots + b_{n}\log\gamma_{n},
$$
with $b_{1}, \ldots , b_{n} \in \mathbb{Z}$ and with the parameter 
$$B = \max \{1 , \max\{b_{j}A_{j}/A_{n}: \  1\leq j  \leq n\}\}.$$ 
For brevity we put
$$
\Omega = A_{1} \ldots A_{n},
$$
$$
C(n) = C(n , \chi) = \frac{16}{n!\chi}e^{n}(2n + 1 + 2 \chi)(n + 2) (4n + 4)^{n + 1}\left(\frac{1}{2}en\right)^{\chi}  ,
$$
$$
C_{0} = \log (e^{4.4n + 7}n^{5.5}d^{2}\log (en)),
$$
$$
W_{0} = \log(1.5eBd\log(ed)).
$$
 The following is the main result of \cite{Mat2}.
 \begin{prop}[Matveev \cite{Mat2}]\label{mat}
If $\log\gamma_{1} , \ldots , \log\gamma_{n}$ are linearly independent over $\mathbb{Z}$ and
$b_{n} \neq 0$, then 
$$
\log|\mathfrak{L}| > -C(n) C_{0} W_{0}d^{2}\Omega.
$$
\end{prop}
\section{Steps of the Proof of Theorem \ref{main}}\label{SoP}

Suppose that $(x , y)$ is an integral solution to (\ref{1.2}). We will assume that $F$ is monic, as we may. Then we have
$$
(x - \alpha_{1} y) (x - \alpha_{2} y) \ldots (x - \alpha_{n}y) = \pm 1.
$$
Therefore, for some $ \alpha \in \{\alpha_{1}, \alpha_{2}, \ldots, \alpha_{n}\} $,
$$
\left| x - \alpha y\right| \leq  1.
$$

\textbf{ Definition}. We say the pair of solution $(x , y)$ is {\emph related} to $\alpha$ if
$$ \alpha \in \{\alpha_{1}, \alpha_{2}, \ldots, \alpha_{n}\} $$
and
$$
\left| x - \alpha y\right| = \min_{1\leq j\leq n} \left| x - \alpha_{j} y\right|.
 $$

Let $F(x , y)$ be a binary form of degree $n \geq 5$, discriminant $D$, with $|D| >D_{0}$ and  Mahler measure $M(F)$, where $D_{0}$ is an effectively computable constant depending only on $n$ (see the statement of Theorem \ref{main}). We will assume that all coefficients of $F$ are integer and $F(x , 1) = 0$ has $r$ real roots and $2s$ non-real roots ($r+2s = n$). 
Here we describe briefly the steps of our proof to the main result of this manuscript,
Theorem \ref{main}. 

In the following  steps, we fix a root of $F(x , 1) = 0$ and estimate the number of solutions related to that root from above.
Let $\alpha$ be a complex  root of $F(x, 1) = 0$ and $\bar{\alpha}$ be its complex conjugate. For integers $x$ and  $y$ we have
$$
\left| x - \alpha y\right| = \left| x - \bar{\alpha} y\right|.
$$
 Hence,
 a solution $(x , y)$ of (\ref{1.2}) is related to $\alpha$ if and only if it is related to $\bar{\alpha}$.   It is, therefore, sufficient to count the number of solutions related to one of $\alpha$ and $\bar{\alpha}$.

\begin{prop}\label{Grp}
For binary form $F(x , y)$ with integer coefficients and degree $n$, let $\alpha$ be a non-real root of $F(x , 1) = 0$. If a pair of integer $(x , y)$ satisfies
$F(x , y) = \pm 1$ and is related to $\alpha$ then
\begin{equation}\label{AG}
|y| \leq \frac{(n+1)  2^{\frac{(n-1)^2}{n}}}{\left(\sqrt{3} \left|D\right|\right)^{1/n}} M(F)^{3-3/n},
\end{equation}
\end{prop}
\begin{proof}
 Let $\alpha = \mathfrak{r} +i \mathfrak{t}$, with $\mathfrak{t} \neq 0$, be a non-real root of $F(x, 1) = 0$. If a solution $(x , y)$ of (\ref{1.2}) is related to $\alpha$ then  $\bar{\alpha}$, the complex conjugate of $\alpha$ is also a root of  $F(x, 1) = 0$ and we have
 $$
 \left| \frac{x}{y} - \alpha \right| = \frac{\left| \frac{x}{y} - \alpha \right| + \left| \frac{x}{y} - \bar{\alpha} \right| }{2} \geq\frac{ \left| \alpha - \bar{\alpha} \right|}{2}.$$
 Moreover, if $\beta \neq \alpha$ is a root of $F(x , 1) = 0$ then
 $$
 \left| \frac{x}{y} - \beta \right| \geq\frac{ \left| \frac{x}{y} - \alpha \right| +  \left| \frac{x}{y} - \beta \right|}{2} \geq  \frac{\left| \beta - \alpha \right|}{2}.
 $$
 Thus 
     \begin{eqnarray*}
 \frac{1}{|y|^n} & = & \left| \frac{x}{y} - \alpha \right|  \prod_{\alpha_{i} \neq \alpha} \left| \frac{x}{y} - \alpha_{i}\right|  \\ \nonumber
 &\geq& \frac{\left|\alpha - \bar{\alpha}\right|}{2}  \prod_{\alpha_{i} \neq \alpha}\frac{ \left| \alpha - \alpha_{i}\right|}{2}  \\ \nonumber
& = &  \left|\alpha - \bar{\alpha}\right| \left|f' (\alpha)\right| 2^{-n}.
  \end{eqnarray*}
  By Lemma \ref{mahl5},
$$
  \left|\alpha - \bar{\alpha}\right| \geq \sqrt{3} (n+1)^{-n} M(F)^{-n+1}.
$$
This, together with Lemma \ref{esug}, shows that 
$$
 \frac{1}{|y|^n} \geq \sqrt{3} (n+1)^{-n}2^{-(n-1)^2} \frac{\left|D\right|}{M(f)^{3n -3}} .
$$
This completes our proof.
\end{proof}

 Repeating an argument of Stewart  \cite{Ste} and using our assumption  that  absolute value of the discriminant of $F$ is large in terms of its degree, in Section \ref{LS} we will show that there are at most $5 (r+s) $ solutions $(x , y)$ with  $0 < y \leq M(F)^2$.

 Lemma \ref{SC1}  and \ref{SC2} give an upper bound  $2r + s$ for  the number of solutions $(x , y)$ with $M(F)^2 < y < M(F)^{1+ (n-1)^2}$.  To prove Lemma \ref{SC1} we will appeal to a classical  inequality of Lewis and Mahler
 (see Lemma \ref{3S}).

 For a non-real root $\alpha$  of $F(x, 1) = 0$, Proposition  \ref{Grp} says that we  only need to count the solutions $(x , y)$ related to $\alpha$ with 
 $$|y| \leq \frac{(n+1)  2^{(n-1)^2/n}}{\sqrt{3} \left|D \right|^{1/n}} M(F)^{3-3/n}.$$
The solutions with larger $y$ must be related to a real root of $F(x , 1) = 0$.

Our approach  to count the number of possibly remaining solutions differs from the approach of Bombieri-Schmidt \cite{Bom} and Stewart \cite{Ste}. In Section \ref{LC}, we will define 
a logarithmic map $\phi(x , y)$. Some geometric properties of this curve lead us to obtain an exponential gap principle in Section \ref{EGP}.
This new type of gap principle, together with Baker theory of linear forms in logarithms (see Proposition \ref{mat}),
 will be used in Section \ref{LFL} to establish an upper bound $2r$ for  the number of solutions $(x , y)$ with $y \geq M(F)^{1+ (n-1)^2}$.

For some technical reasons, particularly to estimate quantities in Proposition \ref{mat} while counting the number of solutions $(x , y) $ with $y \geq M(F)^{1+ (n-1)^2}$, we will need to exclude a set of solutions from our search. This set is   called  $\mathfrak{A}$ and is defined in section \ref{LS}. The set $\mathfrak{A}$ contains $2r+2s-2$ ``small'' solutions.  

Hence, under the assumption of Theorem \ref{main},   there can not exist more than $11r + 4s -2$ to  equation (\ref{1.2}).


\section{The Logarithmic Curve $\phi(x , y)$} \label{LC}

  In order to count the number of ``large'' solutions to $F(x , y) = 1$, many  mathematicians including Bombieri and Schmidt \cite{Bom} and Stewart \cite{Ste} followed and refined a general method  inaugurated by Siegel and Mahler. The general line of attack to the problem of counting ``large" solutions deals rather efficiently with solutions $x$, $y$ to $F(x , y) =  1$, provided that $\max (|x| , |y|)$ is larger than a certain power of the height of $F$. We will, in contrast, associate a transcendental curve $\phi(x , y)$ to the binary form $F(x , y)$. However, the reason in success of both our method and the more classical method of Siegel and Mahler lies in the fact that $\frac{x}{y}$ is a good approximation to a root of the equation $F(x , 1) = 0$ when either $x$ or $y$ is large enough.

Let $D$ be the discriminant of the binary form $F(x , y)$ and $f(x) = F(x , 1)$. Define, for $m \in \{1, 2, \ldots, n\}$,
\begin{equation}\label{fi}
\phi_{m}(x , y) = \log \left| \frac{D^{\frac{1}{n(n-2)}}(x - y\alpha_{m})}{\left( f'(\alpha_{m})\right)^{\frac{1}{n-2} }}\right|
\end{equation}
and
\begin{equation}\label{fgs}
\phi(x , y) = \left( \phi_{1}(x , y) , \phi_{2}(x , y), \ldots , \phi_{n}(x , y) \right).
\end{equation}

We will estimate the size of $f'(\alpha_{m})$ from below in order to give an upper bound on the size of $\phi(x , y)$.

\begin{lemma}\label{lem10}
Suppose that $F$ is a monic binary form satisfying the conditions  in Theorem \ref{main}.
Then $(1 , 0)$ is a solution to the equation $\left|F(x , y)\right| = 1$ and 
$$
 \left\| \phi(1 , 0) \right\| \leq n \log \left( |D|^{\frac{1}{n(n-2)}} M(F)^{\frac{2n -2}{n-2}}\right),
 $$
\end{lemma}
\begin{proof}
By the definition of $\phi$ in (\ref{fgs}),
$$
\left\| \phi(1 , 0) \right\| \leq \sum_{m=1}^{n} \log \left| \frac{D^{\frac{1}{n(n-2)}}}{\left| f'(\alpha_{m})\right|^{\frac{1}{n-2} }}\right|
$$
Lemma \ref{esug}  estimates $\left| f'(\alpha_{m})\right|^{\frac{1}{n-2} }$ as follows,
$$
|f'(\alpha_{m})| \geq 2^{-(n-1)^2} \frac{\left|D\right|}{M(F)^{2n -2}}.
$$
Since  $D_{F}$ is large, definitely larger than $2^{-(n-1)^2}$, we have
$$
|f'(\alpha_{m})| \geq  \frac{1}{M(F)^{2n -2}}.
$$
This completes  our proof.
\end{proof}
  
  \begin{lemma}\label{lem1}
Suppose that $(x ,y)$ is a solution to the equation $\left|F(x , y)\right| = 1$ for the binary form $F$ in Theorem \ref{main}. Suppose that
$$
\left| x - \alpha_{i} y\right| = \min_{1\leq j\leq n} \left| x - \alpha_{j} y\right|.
 $$
Then 
$$
 \left\| \phi(x , y) \right\| \leq \frac{(n+1)^2}{4} \log \frac{1}{ \left| x - \alpha_{i} y\right|} + n \log \left( |D|^{\frac{1}{n(n-2)}} M(F)^{\frac{2n -2}{n-2}}\right),
 $$
where $\| . \|$ is the Euclidean norm.
\end{lemma}
\begin{proof}
Since $|F(x , y)| = \prod_{1\leq j\leq n} \left| x - \alpha_{j} y\right| = 1$
 and 
$
\left| x - \alpha_{i} y\right| = \min_{1\leq j\leq n} \left| x - \alpha_{j} y\right|,
 $
we have 
$\left| x - \alpha_{i} y\right| \leq 1$. 
Let us assume that 
$$
\left| x - \alpha_{s_{j}} y\right| \leq 1, \qquad  \textrm{for}  \  1 \leq j \leq p
 $$
 and 
 $$
\left| x - \alpha_{b_{k}} y\right| > 1, \qquad \textrm{for} \  1\leq k \leq n-p,
$$
where $1\leq p, s_{j} , b_{k} \leq n$.
Since 
  $$
\left| x - \alpha_{i} y\right| = \min_{1\leq j\leq n} \left| x - \alpha_{j} y\right|,
 $$
we  have
  $$
\left|\log \left| x - \alpha_{s_{j}} y\right| \right| \leq \left| \log  \left| x - \alpha_{i} y\right| \right|. 
 $$ 
We also have 
$$
\prod_{k} \left| x - \alpha_{b_{k}} y\right| = \frac{1} {\prod_{j} \left| x - \alpha_{s_{j}} y\right|} .
$$
Therefore, for any $ 1 \leq k \leq n-p$, we have
$$
\log \left| x - \alpha_{b_{k}} y\right| \leq p \log  \frac{1}{\left| x - \alpha_{i} y\right|}.
 $$
From here and the definition of $\phi(x , y)$ (see (\ref{fgs})), we conclude that
  \begin{eqnarray*}
 \left\| \phi(x , y) \right\| & \leq&\sum_{m=1}^{n} \log \left| \frac{D^{\frac{1}{n(n-2)}}}{\left| f'(\alpha_{m})\right|^{\frac{1}{n-2} }}\right| +(n -p)p \left| \phi_{i}( x , y) \right| +p \left| \phi_{i}( x , y) \right|  \\ \nonumber
  & = & \sum_{m=1}^{n} \log \left| \frac{D^{\frac{1}{n(n-2)}}}{\left| f'(\alpha_{m})\right|^{\frac{1}{n-2} }}\right|+
\left((n+1)p - p^2\right) \left| \phi_{i}( x , y) \right|.
 \end{eqnarray*}
  The function $f(p) = (n+1)p - p^2$ assumes its maximum value $\frac{(n+1)^2}{4}$ at $p = \frac{n+1}{2}$. To  complete the proof we  use our  estimate in Lemma \ref{lem10}.
 \end{proof}

\begin{lemma}\label{lem100}
Let $F$ be an irreducible monic binary form of degree $n$. Suppose that $(x , y)$ is a solution to the Thue equation $F(x , y) = \pm 1$ with $y \geq M(F)^{1+(n-1)^2}$. Then
$$
\left\| \phi(1 , 0) \right\| < \left\| \phi(x , y) \right\|.
$$
\end{lemma}
\begin{proof}
Let $\alpha_{1}$, $\ldots$, $\alpha_{n}$ be the roots of $F(z , 1) = 0$. Then 
$$
(\frac{x}{y} - \alpha_{1}) \ldots (\frac{x}{y} - \alpha_{n}) = \frac{\pm 1}{y^n}.
$$
There must exist a root $\alpha_{j}$ so that $\left|\frac{x}{y} - \alpha_{j}\right| \geq \frac{1}{y}$. By Lemma \ref{esug} and since $y \geq M(F)^{1+(n-1)^2}$, the absolute value of the term $\phi_{j}(x , y)$ alone exceeds $n \log \left( |D|^{\frac{1}{n(n-2)}} M(F)^{\frac{2n -2}{n-2}}\right)$. By Lemma \ref{lem10},
 our proof is complete. 
\end{proof}

   Let $U$ be the unit group  of the algebraic number field $\mathbb{Q}(\alpha)$. We define the mapping $\tau$ on $U$ to be the obvious restriction of the embedding of $\mathbb{Q}(\alpha)$ in $\mathbb{C}^{n}$; i.e. $\tau: u \longmapsto (\sigma_{1}(u), \sigma_{2}(u) \ldots \sigma_{n}(u))$, where $\sigma_{i}(u)$ are algebraic conjugates of $u$. 
 By Dirichlet's unit theorem, we have a sequence of mappings 
\begin{equation}\label{ta}
\tau :  U  \to V  \subset \mathbb{C}^{n}
\end{equation}
and
\begin{equation}\label{lo}
 \log : V   \to \Lambda , 
\end{equation}
  where $\Lambda$ is a $(r+s -1)$-dimensional lattice in $\mathbb{R}^n$ and the mapping $\log$ is defined as follows. \newline
For $(x_{1}, \ldots , x_{n}) \in  V $,  let
 $$
 \log(x_{1} , x_{2} , \ldots , x_{n}) := (\log|x_{1}|, \log|x_{2}|, \ldots,  \log|x_{n}|).
 $$
 
  Suppose that $\{ \lambda_{2}, \ldots ,\lambda_{r+s}\}$ is a system of  fundamental units of $\mathbb{Q}(\alpha)$.  Then $\log\left(\tau(\lambda_{2})\right) , \ldots , \log \left(\tau(\lambda_{r+s})\right)$ form a  basis for the lattice $\Lambda$. Moreover,  every basis for $\Lambda$ is associated with a system of fundamental units of  $\mathbb{Q}(\alpha)$. So we will fix a system of fundamental units $\{ \lambda_{2}, \ldots ,\lambda_{r+s}\}$ so that $\log\left(\tau(\lambda_{2})\right) , \ldots , \log \left(\tau(\lambda_{r+s})\right)$ are respectively first to $r+s-1$-th successive minima of the lattice $\Lambda$ (see \cite{Cas}, for the definition of successive minima).
Therefore,
  $$
  \left\|\log \left( \tau(\lambda_{2})\right)\right\|  \leq  \ldots  \leq \left\| \log\left(\tau(\lambda_{r+s})\right)\right\|,
   $$
where $\| . \|$ is the Euclidean norm.
If  $(x , y)$ is a pair of solution to (\ref{1.2}) then 
 $ \frac{x - \alpha_{i}y}{x - \alpha_{j}y }$
  is a unit in $\mathbb{Q}(\alpha_{i} , \alpha_{j})$ and we may write
 \begin{equation}\label{rep}
 \phi(x , y) = \phi(1 , 0) + \sum_{k = 2}^{r+s} m_{k}\log\left(\tau(\lambda_{k})\right), \qquad m_{k} \in \mathbb{Z}.
 \end{equation}

 \section{Layers of Solutions}\label{LS}

As we defined in Section \ref{SoP}, a solution $(x , y)$ is said to be {\emph related} to $\alpha_{i}$ if
$$
\left| x - \alpha_{i} y\right| = \min_{1\leq j\leq n} \left| x - \alpha_{j} y\right|.
 $$

Fix a positive real number $Y_{0}$. Let us first find a bound for the number of solutions $(x , y)$ with $0 < y \leq Y_{0}$.
We may suppose that $F(x , y)$ is a monic form with integral coefficients and  has the smallest Mahler measure among all equivalent monic forms.  Following Stewart \cite{Ste} and Bombieri and Schmidt \cite{Bom}, we will estimate the number of  solutions $(x , y)$ to (\ref{1.2}) for which $0< y \leq Y_{0}$.    For binary form 
$$
F(x , y) = (x - \alpha_{1}y)\ldots  (x - \alpha_{n}y)
$$
put 
$$L_{i}(x , y) = x - \alpha_{i}y
$$
 for $i =1, \ldots, n$. Then
\begin{lemma}\label{S56}
Suppose $F$ is a monic binary form with integral coefficients. Then for every solution $(x , y)$ of (\ref{1.2}) we have
$$
\frac{1}{L_{i}(x , y)} - \frac{1}{L_{j}(x , y)} = (\beta_{j} - \beta_{i}) y,
$$
where $\beta_{1}$,\ldots, $\beta_{n}$ are such that the form
$$
J(u , w) = (u - \beta_{1}w)\ldots (u - \beta_{n}w)
$$
is equivalent to $F$.
\end{lemma}
\begin{proof}
 This is Lemma 4 of \cite{Ste} and Lemma 3 of \cite{Bom}, by taking $(x_{0}, y_{0}) = (1 , 0)$.
\end{proof}

For every  solution $(x , y) \neq (1 , 0)$ of (\ref{1.2}), fix $j = j(x , y)$ with
$$
\left|L_{j}(x , y) \right| \geq 1.
$$
Then, by Lemma \ref{S56},
\begin{equation}\label{S57}
\frac{1}{\left|L_{i}(x , y) \right|} \geq |\beta_{j} - \beta_{i}| |y| - 1.
\end{equation}
For complex conjugate  $\bar{\beta_{j}}$  of $\beta_{j}$, where $j = j(x , y)$,  we also have
$$  
\frac{1}{\left|L_{i}(x , y) \right|} \geq |\bar{\beta_{j}} - \beta_{i}| |y| - 1.
$$
Hence
$$
\frac{1}{\left|L_{i}(x , y) \right|} \geq |\textrm{Re}(\beta_{j}) - \beta_{i}| |y| - 1,
$$
where $\textrm{Re}(\beta_{j})$ is the real part of $\beta_{j}$.
We now choose an integer $m = m(x , y)$ with $|\textrm{Re}(\beta_{j}) - \beta_{j}|\leq 1/2$, and we obtain 
\begin{equation}\label{S58}
\frac{1}{\left|L_{i}(x , y) \right|} \geq \left(|m- \beta_{i}| -\frac{1}{2}\right) |y| - 1,
\end{equation}
for $i = 1,\ldots , n$.

For $1\leq i \leq n$, Let $\frak{X}_{i}$ be the set of solutions to (\ref{1.2}) with $1\leq y \leq Y_{0}$ and $\left|L_{i}(x , y) \right| \leq \frac{1}{2y}$.

\textbf{Remark 1.} When $\alpha_{k}$ and $\alpha_{l}$ are complex conjugates, $\frak{X}_{l} = \frak{X}_{k}$ and therefore 
we only need to consider $r + s$ different sets $\frak{X}_{i}$.

\bigskip

 \textbf{Remark 2.} If a solution $(x , y)$ with $1\leq y \leq Y_{0}$ is related to $\alpha_{i}$ then $(x , y) \in \frak{X}_{i}$.

\bigskip

\textbf{Remark 3.} A solution $(x , y)$ may belong to more than one set $\frak{X}_{i}$.

\begin{lemma}\label{Sl5}
Suppose $(x_{1} , y_{1})$ and $(x_{2} , y_{2})$ are two distinct solutions in $\frak{X}_{i}$ with $y_{1} \leq y_{2}$. Then
$$
\frac{y_{2}}{y_{1}} \geq \frac{2}{7} \max(1 , |\beta_{i}(x_{1} , y_{1}) - m(x_{1} , y_{1})|).
$$
\end{lemma}
\begin{proof}
This is Lemma 5 of \cite{Ste} and Lemma 4 of \cite{Bom}.
\end{proof}

\begin{lemma}\label{Sl6}
Suppose $(x , y)$ is a solution to (\ref{1.2}) with $y > 0$ and $\left|L_{i}(x , y) \right|> \frac{1}{2y}$. Then 
$$
|m(x , y) - \beta_{i}(x , y)| \leq \frac{7}{2}.
$$
\end{lemma}
\begin{proof}
This is Lemma 6 of \cite{Ste}.
\end{proof}

By Lemma \ref{S56} the form 
$$
J(u , w) = (u - \beta_{1}w)\ldots (u - \beta_{n}w)
$$
is equivalent to $F(x , y)$ and therefore the form
$$
\hat{J}(u , w) = (u - (\beta_{1}-m) w)\ldots (u - (\beta_{n}- m)w)
$$
is also equivalent to $F(x , y)$. Therefore, since we assumed that $F$ has the smallest Mahler measure among its equivalent forms, we get
\begin{equation}\label{Spre60}
\prod_{i=1}^{n} \max(1, |\beta_{1}(x , y)-m(x , y)|) \geq M(F). 
\end{equation}
For each set $\frak{X}_{i}$ that is not empty,  let  $(x^{(i)} , y^{(i)})$ be the element with the largest value of $y$.
Let $\frak{X}$ be the set of solutions of (\ref{1.2}) with $1 \leq y \leq Y_{0}$ minus the elements $(x^{(1)} , y^{(1)})$, \ldots, $(x^{(r+s)} , y^{(r+s)})$.
Suppose that, for integer $i$, the set   $\frak{X}_{i}$ is non-empty. Index the elements of $\frak{X}_{i}$ as 
$$(x_{1}^{(i)}, y_{1}^{(i)}), \ldots, (x_{v}^{(i)}, y_{v}^{(i)}),$$
 so that $y_{1}^{(i)} \leq \ldots \leq y_{v}^{(i)}$ (note that $(x_{v}^{(i)}, y_{v}^{(i)}) = (x^{(i)} , y^{(i)})$). By Lemma \ref{Sl5}
\begin{equation*}
\frac{2}{7} \max\left(1, \left|\beta_{i}(x_{k}^{(i)}, y_{k}^{(i)})\right|\right) \leq \frac{y_{k+1}^{(i)}}{y_{k}^{(i)}}
\end{equation*}
for $k = 1 \ldots, v-1$. Hence
$$
\prod_{(x , y) \in  \frak{X} \bigcap \frak{X_{i}} }\frac{2}{7} \max\left(1, \left|\beta_{i}(x_{k}^{(i)}, y_{k}^{(i)})\right|\right) \leq Y_{0}.
$$ 
For $(x , y)$ in $\frak{X}$ but not in  $\frak{X_{i}}$ we have,  by Lemma \ref{Sl6},
$$
\frac{2}{7} \max\left(1, \left|\beta_{i}(x_{k}^{(i)}, y_{k}^{(i)})\right|\right) \leq 1.
$$
Thus
\begin{equation*}\label{S59}
\prod_{(x , y) \in  \frak{X}}\frac{2}{7} \max\left(1, \left|\beta_{i}(x_{k}^{(i)}, y_{k}^{(i)})\right|\right) \leq Y_{0}.
\end{equation*}
Let $|\frak{X}|$ be the cardinality of $\frak{X}$. Comparing the above inequality with (\ref{Spre60}), we obtain 
\begin{equation}\label{S60}
\left( \left(\frac{2}{7}\right)^{n} M(F)\right)^{ |\frak{X}|} \leq Y_{0}^{r+s},
 \end{equation}
for we have $r+s$ different $\frak{X_{i}}$ .
Therefore, by (\ref{mahD5}), we have
$$
\left(\frac{2}{7}\right)^{n} M(F) \geq M(F) ^{\theta}.
$$ 
Here $\theta = \theta(D)$ may be taken equal to $\frac{1}{2}$, for the discriminant $D$ is assumed to be very large. 
From here and by (\ref{S60}),
$$
|\frak{X}| \leq \frac{(r+s) \log Y_{0}}{\theta \log M(F)}.
$$
Thus, when $Y_{0} = M(F)^2$ and $D_{F}$ is large enough,  we have $|\frak{X}|< 4 (r+s)$. Consequently, there are at most $5(r+s)$ solutions $(x , y)$ with  $0 < y \leq M(F)^2$.   We should remark here that we repeat Stewart's \cite{Ste} approach for counting solutions with small $y$ and no improvement has taken place  in estimating $\theta$. The reason that our value for $\theta$ is smaller  is that we are working with forms with larger discriminant.

In order to count the number of solutions $(x , y)$ with $M(F)^2 < y < M(F)^{1+ (n-1)^2}$, we will need the following refinement of an inequality of Lewis and Mahler:
\begin{lemma}\label{3S}
Let $F$ be a binary form of degree $n \geq 3$ with integer coefficients and  nonzero discriminant $D$. For every pair of integers $(x , y)$ with $ y \neq 0$
$$
\min_{\alpha} \left| \alpha - \frac{x}{y} \right| \leq \frac{2^{n-1} n^{n-1/2} \left(M(F)\right)^{n-2} |F(x , y)|}{|D|^{1/2} |y|^n},
$$
where the minimum is taken over the zeros $\alpha$ of $F(z , 1)$.
\end{lemma}
\begin{proof}
This is Lemma 3 of \cite{Ste}.
\end{proof}
\begin{lemma}\label{SC1}
Let $F(x , y)$ be a binary form with integgral coefficients, degree $n$ and  discriminant $D$, where $|D| \geq D_{0}(n)$. Suppose that $\alpha_{i}$ is a real root of $F(z , 1) = 0$. 
Then related to $\alpha_{i}$,  there are at most $2$ solutions for equation (\ref{1.2}) in integers $x$ and $y$ with $M(F)^2 < y < M(F)^{1+ (n-1)^2}$.
\end{lemma}
\begin{proof}
Assume that $(x_{1} , y_{1})$, $(x_{2}, y_{2})$ and $(x_{3}, y_{3})$ are three distinct solutions  to (\ref{1.2}) and all related to $\alpha_{i}$ with $ y_{3} > y_{2} > y_{1} > M(F)^2$. By Lemma \ref{3S}, for $j = 1, 2$, we have
$$
\left| \frac{x_{j+1}}{y_{j+1}} - \frac{x_{j}}{y_{j}} \right| \leq \frac{2^{n} n^{n-1/2} \left(M(F)\right)^{n-2}}{|D|^{1/2} |y_{j}|^n}.
$$
Since $(x_{1} , y_{1})$, $(x_{2}, y_{2})$ and $(x_{3}, y_{3})$ are  distinct solutions,  for $j = 1, 2$, we have  $|x_{j+1}y_{j} - x_{j}y_{j+1}|\geq 1$. Therefore,
$$
\left|\frac{1}{y_{j}y_{j+1}} \right|\leq \left| \frac{x_{j+1}}{y_{j+1}} - \frac{x_{j}}{y_{j}} \right| \leq \frac{M(F)^{n-2} }{ |y_{j}|^n}.
$$
This is because we assumed that $|D|$ is large. 
Thus,
\begin{equation}\label{S65}
\frac{y_{j}^{n-1}}{M(F)^{n -2}} \leq y_{j+1}.
\end{equation}
Following Stewart \cite{Ste}, we define $\delta_{j}$, for $j = 1, 2, 3$, by
$$
y_{j} = M(F)^{1+\delta_{j}}.
$$
By (\ref{mahD5}), $M(F) > 1$ and so (\ref{S65}) implies that
$$
(n-1) \delta_{j} \leq \delta_{j+1}.
$$ 
 From here, we conclude that 
$$
y_{3} \geq M(F)^{1+ (n-1)^2}.
$$
In other words, related to each real root $\alpha_{i}$, there are at most $2$ solutions in $x$ and $y$ with $M(F)^2 < y < M(F)^{1+ (n-1)^2}$.
\end{proof}

\begin{lemma}\label{SC2}
Let $F(x , y)$ be a binary form with integral coefficients, degree $n$ and  discriminant $D$, where $|D| \geq D_{0}(n)$. Suppose that $\alpha_{i}$ is a non-real root of $F(z , 1) = 0$. 
Then related to $\alpha_{i}$,  there exists  at most $1$ solution to equation (\ref{1.2}) in integers $x$ and $y$ with $M(F)^2 < y < M(F)^{1+ (n-1)^2}$.
\end{lemma}
\begin{proof}
Assume that $(x_{1} , y_{1})$ and $(x_{2}, y_{2})$  are two distinct solutions  to (\ref{1.2}) and all related to $\alpha_{i}$, a non-real root of $F(z , 1) = 0$, with $ y_{2} > y_{1} > M(F)^2$. Similar to (\ref{S65}) in the proof of Lemma \ref{SC1}, we have
$$
\frac{y_{1}^{n-1}}{M(F)^{n -2}} \leq y_{2}.
$$
This contradicts (\ref{AG}), since $y_{1} > M(F)^2$ and $M(F)$ is large. Therefore, related to each non-real $\alpha_{i}$, there is at most $1$ solutions in $x$ and $y$ with $M(F)^2 < y < M(F)^{1+ (n-1)^2}$.
\end{proof}

So we conclude that there are at most $7r + 6s $ solutions $(x , y)$ with $0 < y < M(F)^{1+ (n-1)^2}$ to equation (\ref{1.2}) when $F(z , 1) = 0$ has $r$ real roots and $2s$ non-real ones.

Stewart \cite{Ste} invented the above method to count all solutions with $y > M(F)^2$. He obtained the bound
$$ 
n \left(4 + \frac{\log 331890}{  \log (n -1)}\right)
$$ 
for the number of solutions to (\ref{1.2}) with $y > M(F)^2$ (see page 815 of \cite{Ste}). Our method allows us to save the summand $\frac{\log 331890}{  \log (n -1)}$. This gives us a better bound for binary forms with smaller degree.

 The rest of paper is devoted to count the number of solutions $(x , y)$ with   $y \geq M(F)^{1+ (n-1)^2}$. As we commented in Section \ref{SoP}, we need to consider this case only when we study the solutions $(x , y)$ related to the real roots of $F(x , 1) = 0$.

\begin{lemma}\label{sma}
  For every fixed integer $m$, there are at most $2r+2s -2$ solutions $(x , y)$ to (\ref{1.2}) for which  in (\ref{rep}), $m_{r+s} = m$.
  \end{lemma}
  \begin{proof}
  Let $S$ be the $(r+s-1)$-dimensional affine space of all vectors  $$\phi(1 , 0) + \sum_{i = 2}^{r+s} \mu_{i}\log\left(\tau(\lambda_{i})\right)\qquad (\mu_{i} \in \mathbb{R}).$$
 Let $\mu_{r+s} = m$. Then the points  
  $$
  \phi(1 , 0) + \sum_{i = 2}^{r+s-1} \mu_{i}\log\left(\tau(\lambda_{i})\right)+ m \log\left(\tau(\lambda_{r+s})\right) 
  $$
  form an $(r+s-2)$-dimensional hyperplane $S_{1}$ of $S$. Put $f(t) = F(t, 1)$.  For $t \in \mathbb{R}$, define  $y(t)$ and $x(t)$ as follows:
\begin{eqnarray*}
y(t)&: =& |f(t)|^{-1/n},\\
x(t)&: = &ty(t).
\end{eqnarray*}
Similar to $\phi(x , y)$, we define the curve $\phi(t)$ on $\mathbb{R}$:
\begin{equation*}
\phi(t) = \left( \phi_{1}(t) , \phi_{2}(t), \ldots , \phi_{n}(t) \right),
\end{equation*}
where, for $1\leq m \leq n$,
\begin{equation*}
\phi_{m}(t) = \log \left| \frac{D^{\frac{1}{n(n-2)}}(x(t) - \alpha_{m}y(t))}{\left| f'(\alpha_{m})\right|^{\frac{1}{n-2} }}\right|.
\end{equation*}
Observe that for an integral solution $(x , y)$ to (\ref{1.2}) and $\phi(x , y)$ defined in (\ref{fgs}), we have
$$
\phi(x , y) = \phi\left(\frac{x}{y}\right).
$$

Let $\vec{N} = (N_{1} , \ldots , N_{n}) \in S$ be the normal vector of $S_{1}$.
Then the number  of times that the curve $\phi(t)$ intersects $S_{1}$ equals the number of solutions in $t$  to 
  \begin{equation}\label{N}
  \vec{N}. \phi(t) = 0.
  \end{equation}
We have
$$
\lim_{t \rightarrow \alpha_{i}^{+} }\log|t - \alpha_{i}| =  -\infty
$$
and
$$
\lim_{t \rightarrow \alpha_{i}^{-} }\log|t - \alpha_{i}| =  -\infty .
$$
   Note that if $\alpha_{i}$ is a non-real root of $F(x , 1)$ then $\bar{\alpha_{i}}$, the complex conjugate of $\alpha_{i}$ is also a root and we have
$$
\log|t - \alpha_{i}| = \log|t - \bar{\alpha_{i}}|.
$$
If $\alpha_{1}, \ldots, \alpha_{r}$ are the reals roots and  $\alpha_{r+1}, \ldots, \alpha_{r+s},\alpha_{r+s+1}, \ldots, \alpha_{r+2s}$ are non-real roots with $\alpha_{r+s+k} = \bar{\alpha}_{r+k}$, then
the derivative $ \frac{d}{dt}  \left( \vec{N}. \phi(t) \right)$ can be written as  $\frac{P(t)}{Q(t)}$,
 where $Q(t) = (t -\alpha_{1}) \ldots  (t -\alpha_{r})(t -\alpha_{r+1}) \ldots  (t -\alpha_{r+s})$ and  $P(t)$ is a polynomial of degree  $r+s -1$. Therefore, the derivative has at most $r+s -1$ zeros and consequently, the equation (\ref{N}) can not have more than $2r+2s -2$ solutions.
 \end{proof}

\textbf{Definition of the set $\mathfrak{A}$}.
 Assume that equation (\ref{1.2}) has more than $2r+2s -2$ solutions. Then we can list $(1 , 0)$ and $2r+2s-3$ other solutions $(x_{i} , y_{i})$ ($1 \leq i \leq 2r+2s-3$), so that $\left\| \phi(x_{i} , y_{i})\right\|$ are the smallest among all $\left\| \phi(x , y)\right\|$, where $(x , y)$ varies over all non-trivial pairs of solutions. We denote  the set of all  these $2r+2s -2$ solutions  by $\mathfrak{A}$.

The important property of $\mathfrak{A}$ is that for every solution $(x_{0} , y_{0}) \in \mathfrak{A}$ and every solution $(x , y) \not \in \mathfrak{A}$  to (\ref{1.2}) with $y \geq M(F)^{1+(n-1)^2}$, by Lemma \ref{lem100} and the definition, we have 
  $$
\left\| \phi(x_{0} , y_{0})\right\|\leq \left\| \phi(x , y)\right\|.
$$
  
  \begin{cor}\label{m}
  Let $(x , y) \not \in \mathfrak{A}$ be a  solution to (\ref{1.2}) with $y \geq M(F)^{1+(n-1)^2}$. Then 
  $$
  \left\|\log \left( \tau(\lambda_{2})\right)\right\|  \leq  \ldots  \leq \left\| \log\left(\tau(\lambda_{r+s})\right)\right\| \leq 2\left\| \phi(x , y)\right\|.
   $$
     \end{cor}
     \begin{proof}
     Since we have assumed that 
$  \left\|\log \left( \tau(\lambda_{2})\right)\right\|  \leq  \ldots  \leq \left\| \log\left(\tau(\lambda_{r+s})\right)\right\|
   $, it is enough to show that   $\left\|\log\left(\tau(\lambda_{r+s})\right)\right\| \leq 2 \left\| \phi(x , y)\right\|$. By Lemma \ref{sma}, there is at least one small solution $(x_{0} , y_{0}) \in\mathfrak{A}$ so that 
   $$
   \phi(x , y) - \phi(x_{0} , y_{0}) = \sum_{i = 2}^{r+s} k_{i}\log\left(\tau(\lambda_{i})\right)  , 
  $$ 
   with $k_{n} \neq 0$. Since $\{ \log\left(\tau(\lambda_{i})\right)\}$ is a reduced basis for the lattice $\Lambda$ in (\ref{lo}), by Lemma \ref{lem100} and from the definition of  $\mathfrak{A}$ we conclude that 
       $$
       \left\|\log\left(\tau(\lambda_{r+s})\right)\right\| \leq \left\| \phi(x , y) - \phi(x_{0} , y_{0})\right\| \leq  2\left\|\phi(x , y)\right\|.
       $$
   \end{proof}

\begin{lemma}\label{Dr}
Suppose $(x , y) \not \in \mathfrak{A}$. Then
  $$
  \left\| \phi(x , y) \right\| \geq \frac{1}{2} \log \left(\frac{|D|^{\frac{1}{n(n-1)}}}{2}\right).
  $$
   \end{lemma}
   \begin{proof}
   Let  $(x' , y') \in \mathfrak{A}$ be a pair of solutions to equation (\ref{1.2}) and $\alpha_{i}$ and $\alpha_{j}$ be two distinct roots of the polynomial $F(x , 1)$. We have
\begin{eqnarray*}
\left| e^{\phi_{i}(x' , y') - \phi_{i}(x , y)} - e^{\phi_{j}(x' , y') - \phi_{j}(x , y)}\right| & =& \left|\frac{x' - y'\alpha_{i}}{x - y\alpha_{i}} - \frac{x' - y'\alpha_{j}}{x - y\alpha_{j}}\right| \\
& = & \frac{\left|\alpha_{i} - \alpha_{j}\right| \left|xy' - yx'\right|}{|x - y\alpha_{i}||x - y\alpha_{j}|}\\
&\geq &\frac{\left|\alpha_{i} - \alpha_{j}\right| }{|x - y\alpha_{i}||x - y\alpha_{j}|}.
\end{eqnarray*}
  The last inequality follows from the fact that  $ \left|xy' - yx'\right|$ is a non-zero integer. Since $|\phi_{i}| < \|\phi\|$ and $\left\| \phi(x' , y') \right\| < \left\| \phi(x , y) \right\|$, we may conclude
\begin{eqnarray*}
\left( 2e^{2\left\| \phi(x , y) \right\|}\right)^{\frac{n(n-1)}{2}} & \geq & \prod_{1\leq i < j \leq n}\left| e^{\phi_{i}(x' , y') - \phi_{i}(x , y)} - e^{\phi_{j}(x' , y') - \phi_{j}(x , y)}\right|\\
&\geq & \prod_{1 \leq i < j\leq n}\left|\frac{x' - y'\alpha_{i}}{x - y\alpha_{i}} - \frac{x' - y'\alpha_{j}}{x - y\alpha_{j}}\right| \\
&\geq& \prod_{1 \leq i < j\leq n}\frac{\left|\alpha_{i} - \alpha_{j}\right| }{|x - y\alpha_{i}||x - y\alpha_{j}|}=
\sqrt{|D|}.
\end{eqnarray*}
   \end{proof}

\section{Distance Functions}\label{Geo}
 
 Suppose that $(x , y) \neq (1 , 0)$ is a solution to (\ref{1.2}) and let $t = \frac{x}{y}$. We have
$$
\phi(x , y) = \phi(t) = \sum_{i=1}^{n} \log \frac{|t - \alpha_{i}|}{\left| f'(\alpha_{i})\right|^{\frac{1}{n-2}}} \bf{b_{i}},
$$
where,
$$
\mathbf{b_{i}} = 
\frac{1}{n}(-1, \ldots, -1, n-1, -1, \ldots , -1).
$$
Without loss of generality, we will suppose that  the pair of solution $(x , y)$ is related to $\alpha_{n}$;
$$
\left| x - \alpha_{n} y\right| = \min_{1\leq j\leq n} \left| x - \alpha_{j} y\right|.
$$
We may write
\begin{equation}\label{E}
\phi(x , y) = \phi(t) = \sum_{i=1}^{n-1} \log \frac{|t - \alpha_{i}|}{\left| f'(\alpha_{i})\right|^{\frac{1}{n-2}}} \mathbf{ c_{i}}
 + E_{n}\mathbf{b_{n}},
\end{equation}
where, for $1\leq i \leq n-1$,
\begin{equation}\label{16*}
{\bf c_{i}} = {\bf b_{i}} 
+ \frac{1}{n-1} {\bf b_{n}} , \quad 
 E_{n} = \log \frac{\left| t - \alpha_{n} \right|}{\left|f'(\alpha_{n})\right|^{\frac{1}{n-2}}} - \frac{1}{n-1} \sum_{i=1}^{n-1} \log \frac{|t - \alpha_{i}|}{\left| f'(\alpha_{i})\right|^{\frac{1}{n-2}}} 
\end{equation}
One can easily observe that, for   $1\leq i \leq n$,
\begin{equation}\label{ci}
\mathbf{c_{i}} \perp \mathbf{b_{n}}, \  \textrm{and}\  
 \|\mathbf{c_{i}} \| = \frac{\sqrt{n^2 - 3n +2}}{n-1}.
\end{equation}

\begin{lemma}\label{gp1}
Let
$$
\mathbf{L_{n}} = \left\{ \sum_{i=1}^{n-1} \log \frac{|\alpha_{n} - \alpha_{i}|}{\left| f'(\alpha_{i})\right|^{\frac{1}{n-1}}} \mathbf{c_{i}} + z \mathbf{b_{n}} , \quad z \in \mathbb{R} \right\}.
$$
Suppose that $(x , y)$ is a  solution to (\ref{1.2}) with
$$
\left| x - \alpha_{n} y\right| = \min_{1\leq j\leq n} \left| x - \alpha_{j} y\right|\  \textrm{and}\, \,    y > M(F)^{1+ (n-1)^2} .
 $$ 
Then the distance between $\phi(x, y)$ and the line $\mathbf{L_{n}}$ is less than 
$$
 \frac{1}{M(F)^{n(n -1)}}\exp\left(\frac{-4\left\| \phi(x , y) \right\|}{(n+1)^2}\right).
$$
\end{lemma}
\begin{proof}
The distance between $\phi(x, y)$ and $\mathbf{L_{n}}$ is equal to
$$
\left\|  \sum_{i=1}^{n-1} \log \frac{|t- \alpha_{i}|}{|\alpha_{n} - \alpha_{i}|} \mathbf{c_{i}} \right\|,
$$
where $t= \frac{x}{y}$. We will show that when  $i \neq n-1$,
$$
\left| \log \frac{ |t - \alpha_{i}|} {|\alpha_{n} - \alpha_{i}|} \right| <   \frac{ |t- \alpha_{n}|}{ \min _{i\neq j}\{\left| \alpha_{j} - \alpha_{i} \right| \}},
 $$
 We will consider two cases $ |t - \alpha_{i}| > |\alpha_{n} - \alpha_{i}|$ and $ |t - \alpha_{i}| \leq |\alpha_{n} - \alpha_{i}|$. First assume that  $ |t - \alpha_{i}| > |\alpha_{n} - \alpha_{i}|$. We have
\begin{equation*}
\left| \log \frac{ |t - \alpha_{i}|} {|\alpha_{n} - \alpha_{i}|}  \right|  =  \log \frac{ |t - \alpha_{i}|} {|\alpha_{n} - \alpha_{i}|}  \leq  \log \left(\frac{ |t- \alpha_{n}|} {|\alpha_{n} - \alpha_{i}|} + 1\right)  <    \frac{|t -\alpha_{n}|}{ |\alpha_{i}- \alpha_{n}|}.
 \end{equation*}
 Now assume that  $ |t - \alpha_{i}| \leq |\alpha_{n} - \alpha_{i}|$. Then 
\begin{equation*}
\left| \log \frac{ |t - \alpha_{i}|} {|\alpha_{n} - \alpha_{i}|}  \right|  =  \log \frac{ |\alpha_{n} - \alpha_{i}|} {|t - \alpha_{i}|}  \leq  \log \left(\frac{ |t- \alpha_{n}|} {|t - \alpha_{i}|} + 1\right)  <   \frac{|t -\alpha_{n}|}{ |\alpha_{i}- t|}.
 \end{equation*}
Note that, since we assumed $t$ is closer to $\alpha_{n}$,
$$
 |\alpha_{i}- t| \geq \frac{|\alpha_{i} -t | + |\alpha_{n} - t|}{2} \geq \frac{|\alpha_{i} - \alpha_{n} |}{2}.
$$
 Hence, we obtain
\begin{equation}\label{tmt}
\left| \log \frac{ |t - \alpha_{i}|} {|\alpha_{n} - \alpha_{i}|} \right| <  2 \frac{ |t- \alpha_{n}|}{m},
 \end{equation}
where $m = \min _{i\neq j}\{\left| \alpha_{j} - \alpha_{i} \right| \}$.  
 This, together with   (\ref{ci}), gives
\begin{equation*}
\left\|  \sum_{i=1}^{n-1} \log \frac{|t- \alpha_{i}|}{|\alpha_{n} - \alpha_{i}|} \bf{c_{i}} \right\| <  \frac{2  \sqrt{n (n^2 - 3n +2)}}{n-1} \, \frac{ |u|}{m},
\end{equation*}
where $u = t - \alpha_{n}$.
Using (\ref{mahD5}), we obtain
 \begin{equation}\label{d1}
\left\|  \sum_{i=1}^{n-1} \log \frac{|t- \alpha_{i}|}{|\alpha_{n} - \alpha_{i}|} \bf{c_{i}} \right\| <  \frac{2\, M(F)^{n-1} (n+1)^{n}  \sqrt{n (n^2 - 3n +2)}}{\sqrt{3}(n-1)} \, |u|.
\end{equation}
We shall estimate $|u|$ now. From Lemma \ref{lem1} we have
$$
\left\| \phi(x , y) \right\| - n \log \left( |D|^{\frac{1}{n(n-2)}} M(F)^{\frac{2n -2}{n-2}}\right)
\leq \frac{(n+1)^2}{4} \log \frac{1}{ \left| x - \alpha_{n} y\right|} ,
 $$
which implies
$$
  \log |yu| < \frac{-4\left\| \phi(x , y) \right\|}{(n+1)^2} +\frac{4 n}{(n+1)^2} \log \left( |D|^{\frac{1}{n(n-2)}} M(F)^{\frac{2n -2}{n-2}}\right).
 $$
Therefore,
$$
|u| < \exp\left(\frac{-4\left\| \phi(x , y) \right\|}{(n+1)^2}\right) \frac{\exp \left(\frac{4 n}{(n+1)^2} \log \left( |D|^{\frac{1}{n(n-2)}} M(F)^{\frac{2n -2}{n-2}}\right)\right)}{|y|}
$$
Comparing this with (\ref{d1}), since we took $n \geq 5$ and $|y| > M(F)^{1+ (n-1)^2}$, our proof is complete.
\end{proof}

For $3$ distinct roots  of $F(x , 1) = 0$, say $\alpha_{i}$, $\alpha_{j}$ and $\alpha_{n}$, let us define 
  $$
T_{i , j}(t) := \log \left|\frac{(t - \alpha_{i}) (\alpha_{n} - \alpha_{j})}{(t - \alpha_{j}) (\alpha_{n} -\alpha_{i})}\right|,
  $$
 so that for a pair of solution $(x , y) \neq (1 , 0)$,
  \begin{eqnarray}\label{T}
  T_{i,j} (x , y) = T_{i,j}(t) & = &  \log \left|\frac{\alpha_{n} - \alpha_{i}}{\alpha_{n} -\alpha_{j}}\right| + \log \left|\frac{t - \alpha_{j}}{t - \alpha_{i} }\right| \nonumber \\ 
  & = &  \log |\lambda_{i,j}| + \sum_{k=2}^{r+s} m_{k}\log  \left|\frac{\lambda_{k}}{\lambda'_{k}}\right| ,
  \end{eqnarray}
  where $ t = \frac{x}{y}$ , $\lambda_{i,j} =  \frac{ \alpha_{n} - \alpha_{i}}{\alpha_{n} -\alpha_{j}} $ , $m_{k} = m_{k}(x , y) \in \mathbb{Z}$, and for $2\leq k \leq r+s$, $\lambda_{k}$ are the fundamental units of number field $\mathbb{Q}(\alpha_{i})$ and $\sigma(\lambda_{k}) =\lambda'_{k}$ are the fundamental units of the number field $\mathbb{Q}(\alpha_{j})$ and index $\sigma $ is the $\mathbb{Q}$-isomorphism  from 
$\mathbb{Q}(\alpha_{i})$ to $\mathbb{Q}(\alpha_{j})$ such that $\sigma(\alpha_{i}) = \alpha_{j}$. 
The function $T(x , y)$ cries out to be treated by Baker's theory of linear forms in logarithms.  For this we will wait till the very last part of the paper,  Section \ref{LFL}, where we estimate $ \left|T_{i , j}\right|$ from below.  The following lemma gives an upper bound upon $ \left|T_{i,j}\right|$.
 \begin{lemma}\label{Tu}
Let $(x , y)$ be a pair of solution to (\ref{1.2}) with $|y| > M(F)^{1+ (n-1)^2}$. Then there exists a pair $(i , j)$ for which 
 $$
\left|T_{i, j}(x , y) \right|  < \frac{\sqrt{ \frac{2}{n-2}}}{M(F)^{n (n -1)}} \exp\left(\frac{-4\left\| \phi(x , y) \right\|}{(n+1)^2}\right).$$
  \end{lemma}
\begin{proof}
Let us define 
$$
\beta_{i} = \left \{ \begin{array}{ll}
\alpha_{i} & \mbox{if  $i \leq n-1$} \\
\beta_{i-n+1} & \mbox{ if $i \geq n$.}
\end{array}
\right.
$$
Note that 
\begin{eqnarray}\nonumber
& & \sum_{k=1}^{n-2}\sum_{i=1}^{n-1} \log^2 \left| \frac{(t - \beta_{i}) (\alpha_{n} - \beta_{i+k})}{(\alpha_{n} -\beta_{i}) ( t- \beta_{i+k})}\right|     \\
 & = & 2(n-2) \sum_{i=1}^{n-1} \log^2 \left| \frac{t - \alpha_{i}} {\alpha_{n} -\alpha_{i}} \right| -  4 \sum_{\substack{j\neq i \\ j \neq n}} \log \left| \frac{t - \alpha_{i}} {\alpha_{n} -\alpha_{i}} \right|   \log \left| \frac{t - \alpha_{j}} {\alpha_{n} -\alpha_{j}} \right|     \nonumber  \\ 
 & = &  2(n -2) \sum_{i=1}^{n-1} \log^2 \left| \frac{t - \alpha_{i}} {\alpha_{n} -\alpha_{i}} \right| -  2\sum_{i=1} ^{n-1}\log \left| \frac{t - \alpha_{i}} {\alpha_{n} -\alpha_{i}} \right| \sum_{\substack{j\neq i \\ j \neq n}} \log \left| \frac{t - \alpha_{j}} {\alpha_{n} -\alpha_{j}} \right|       \nonumber  \\ 
 & =&   2(n-2) \sum_{i=1}^{n-1} \log^2 \left| \frac{t - \alpha_{i}} {\alpha_{n} -\alpha_{i}} \right| - 2  \sum_{i=1} ^{n-1}\log \left| \frac{t - \alpha_{i}} {\alpha_{n} -\alpha_{i}} \right| \log \left| \frac{\alpha_{n} - \alpha_{i}} {y^n f'(\alpha_{n})( t- \alpha_{n})(t - \alpha_{i})} \right|       \nonumber  \\ 
 & =&  (2n-2) \sum_{i=1}^{n-1} \log^2 \left| \frac{t - \alpha_{i}} {\alpha_{n} -\alpha_{i}} \right| - 2 \log \left| \frac{1}  {y^n f'(\alpha_{n})( t- \alpha_{n})} \right| \sum_{i=1} ^{n-1}\log \left| \frac{t - \alpha_{i}} {\alpha_{n} -\alpha_{i}} \right|  \nonumber     \\ 
& =& (2n -2) \sum_{i=1}^{n-1} \log^2 \left| \frac{t - \alpha_{i}} {\alpha_{n} -\alpha_{i}} \right| - 2 \log^2 \left| \frac{1}  {y^n f'(\alpha_{n})( t- \alpha_{n})} \right|  \nonumber
  \end{eqnarray}
  On the other hand, it follows 
from the proof of Lemma \ref{gp1} that  the distance between $\phi(x, y)$ and the line 
  $$
{\bf L_{n}} = \sum_{i=1}^{n-1} \log \frac{|\alpha_{n} - \alpha_{i}|}{\left| f'(\alpha_{i})\right|^{\frac{1}{n-1}}} {\bf c_{i}} + z {\bf b_{n}} , \quad z \in \mathbb{R},
$$
 is equal to  $\left\|  \sum_{i=1}^{n-1} \log \frac{|t- \alpha_{i}|}{|\alpha_{n} - \alpha_{i}|} \bf{c_{i}} \right\|$. Further, by the definition of  $\bf{c_{i}}$ in (\ref{16*}), we have
\begin{eqnarray*}
& & \left\|  \sum_{i=1}^{n-1} \log \frac{|t- \alpha_{i}|}{|\alpha_{n} - \alpha_{i}|} \bf{c_{i}} \right\| ^2\\
& = & \left\| \sum_{i=1}^{n-1}\log \left(\frac{|t- \alpha_{i}|}{|\alpha_{n} - \alpha_{i}|}\right) - \frac{1}{n-1}\left| \log \frac{1}  {y^n f'(\alpha_{n})( t- \alpha_{n})} \right|    \bf{e_{i}}\right\|^2 \nonumber \\
& = &  \sum_{i=1}^{n-1}\log^2 \left(\frac{|t- \alpha_{i}|}{|\alpha_{n} - \alpha_{i}|}\right) - \frac{1}{n-1}\left| \log \frac{1}  {y^n f'(\alpha_{n})( t- \alpha_{n})} \right|     \nonumber \\
\\& =& \sum_{i=1}^{n-1} \log^2 \left| \frac{t - \alpha_{i}} {\alpha_{n} -\alpha_{i}} \right| - \frac{1}{n-1}  \log \left| \frac{1}  {y^n f'(\alpha_{n})( t- \alpha_{n})} \right| \sum_{i=1} ^{n-1}\log \left| \frac{t - \alpha_{i}} {\alpha_{n} -\alpha_{i}} \right|    \nonumber     \\ 
\end{eqnarray*}

 where $\{ {\bf e_{i}} \}$ is the standard basis for $\mathbb{R}^{n-1}$. So, there must be a pair  $(i , j)$, for which the following holds:
 \begin{eqnarray*}
 & & \log^2 \left|\frac{(t - \alpha_{i}) (\alpha_{n} - \alpha_{j})}{(t - \alpha_{j}) (\alpha_{n} -\alpha_{i})}\right| \\
 & < &\frac{1}{(n-1)(n-2)} \sum_{k=1}^{n-2}\sum_{i=1}^{n-1} \log^2 \left| \frac{(t - \beta_{i}) (\alpha_{n} - \beta_{i+k})}{(\alpha_{n} -\beta_{i}) ( t- \beta_{i+k})}\right|    =  \\ 
  & = & \frac{2(n -1)}{(n-1)(n-2)} \left\|  \sum_{i=1}^{n-1} \log \frac{|t- \alpha_{i}|}{|\alpha_{n} - \alpha_{i}|} \bf{c_{i}} \right\| ^2 .
  \end{eqnarray*}
  Therefore, by Lemma \ref{gp1}
  $$
\left|T_{i, j}(x , y) \right|=  \left| \log \left|\frac{(t - \alpha_{i}) (\alpha_{n} - \alpha_{j})}{(t - \alpha_{j}) (\alpha_{n} -\alpha_{i})}\right| \right|  < \frac{\sqrt{\frac{2}{n-2}}}{M(F)^{(n -2) (n -3)}}  \exp\left(\frac{-4\left\| \phi(x , y) \right\|}{(n+1)^2}\right).
  $$
\end{proof}


\section{ Exponential Gap Principle}\label{EGP}

Here our goal is to prove
\begin{thm}\label{exg}
Suppose that  $(x_{1} , y_{1})$, $(x_{2} , y_{2})$ and $(x_{3} , y_{3})$ are three pairs of non-trivial solutions to (\ref{1.2}) with 
$$
\left| x_{j} - \alpha_{n} y_{j}\right| \leq 1,
$$
for $j \in \{1 , 2 , 3\}$.  If $r_{1} \leq r_{2} \leq r_{3}$ then
$$
r_{3} >  M(F)^{n (n - 1)} \exp\left(\frac{4r_{1}}{(n+1)^2}\right)\frac{\sqrt{3}}{256}  \, \left(\frac{\log \log n}{\log n}\right)^6,$$
where $r_{j} = \left\| \phi(x_{j} , y_{j}) \right\|$.
\end{thm}
\begin{proof}
Suppose that $(x_{1} , y_{1})$, $(x_{2} , y_{2})$ and $(x_{3} , y_{3})$ are three pairs of non-trivial solutions to (\ref{1.2}). 
We note that  three point $\phi_{1} = \phi(x_{1}, y_{1})$,  $\phi_{2} = \phi(x_{2}, y_{2})$ and  $\phi_{3} =\phi(x_{3}, y_{3})$ form a triangle $\Delta$.  The length of each side of $\Delta$ is less than $2r_{3}$. 
Lemma \ref{gp1} shows that the height of $\Delta$ is at most 
$$
\frac{2}{ M(F)^{n (n -1)}}\, \exp\left(\frac{-4r_{1}}{(n+1)^2}\right).
$$
Therefore, the area of $\Delta$ is less than 
\begin{equation}\label{up}
\frac{4}{ M(F)^{n (n - 1)}}\,   r_{3} \exp\left(\frac{-4r_{1}}{(n+1)^2}\right).
\end{equation}

To estimate the area of $\Delta$ from below, we note that $x - \alpha_{i} y$ is a unit in $\mathbb{Q}(\alpha_{i})$ when $(x , y)$ is a pair of solution to (\ref{1.2}).  This is because
$$
F(x , y) = (x - \alpha_{1} y) (x - \alpha_{2} y) \ldots (x - \alpha_{n}y) = \pm 1.
$$
 Define the vector ${\bf \vec{e}}$ as follows
$$
{\bf \vec{e}} = \phi(x_{1} , y_{1}) - \phi(x_{2} , y_{2}) = \left(\log \left|\frac{x_{1} - \alpha_{1}y_{1}}{x_{2} - \alpha_{1}y_{2}}\right|, \ldots, \log \left| \frac{x_{1} - \alpha_{n}y_{1}}{x_{2} - \alpha_{n}y_{2}}\right| \right) .
$$
Since $ x_{1} - \alpha_{i}y_{1}$ and $x_{2} - \alpha_{i}y_{2}$ are units in $\mathbb{Q}(\alpha_{i})$, by Lemma \ref{hM5} we have
$$
\| {\bf \vec{e}} \| \geq n h(\alpha_{1}) > \frac{1}{4} \left(\frac{\log \log n}{\log n}\right)^3.
$$
Now we can estimate each side of $\Delta$ from below to conclude that the area of the triangle $\Delta$ is greater than 
$$
\frac{\sqrt{3}}{64} \left(\frac{\log \log n}{\log n}\right)^6 .
$$
Comparing this with (\ref{up}) we conclude that
$$
\frac{4}{ M(F)^{n (n -1)}}\, r_{3}  \exp\left(\frac{-4r}{(n+1)^2}\right) > \frac{\sqrt{3}}{64}  \, \left(\frac{\log \log n}{\log n}\right)^6.
$$
 The result is  immediate from here.
\end{proof}

\textbf{Remark}. If all the roots of polynomial $F(x , 1)$ are real then we can use the following lower bound for the size of vector ${\bf \vec{e}}$:
$$
\| {\bf \vec{e}} \| \geq n \log^2 \frac{1 + \sqrt{5}}{2}
$$
(see exercise 2 on page 367 of \cite{Poh}). Now an argument similar to the proof of Theorem \ref{exg} shows that in this case, 
$$
r_{3} >  \frac{ M(F)^{n (n -1)}}{2}\, \exp\left(\frac{4r_{1}}{(n+1)^2}\right)\frac{\sqrt{3}}{8}\, n^2 \log^4 \frac{1 + \sqrt{5}}{2}.
$$


 \section{Linear Forms in Logarithms}\label{LFL}

 Let  $\sigma $ be the $\mathbb{Q}$-isomorphism  from 
$\mathbb{Q}(\alpha_{i})$ to $\mathbb{Q}(\alpha_{j})$ such that $\sigma(\alpha_{i}) = \alpha_{j}$. 
Suppose that there are three solutions $(x_{1}, y_{1})$, $(x_{2} , y_{2})$, $(x_{3} , y_{3})$ to (\ref{1.2})  satisfying the following conditions
$$(x_{l} , y_{l}) \not \in \mathfrak{A},$$
 $$
 y_{l} > M(F)^{1+(n-1)^2}
 $$
and
 $$
\left| x_{l} - \alpha_{n} y_{l}\right|  = \min_{1\leq i \leq n} \left| x_{l} - \alpha_{i} y_{l}\right|   \quad l \in \{1 , 2 , 3 \}.
$$
Assume that $r_{1} \leq r_{2} \leq r_{3} $, where $r_{j} = \left\| \phi( x_{j} , y_{j}) \right\|$.
 We will apply  Matveev's lower bound to 
 \begin{eqnarray*}
  T_{i,j} (x_{3} , y_{3}) = T_{i,j}(t_{3}) & = &  \log \left|\frac{\alpha_{n} - \alpha_{i}}{\alpha_{n} -\alpha_{j}}\right| + \log \left|\frac{t_{3} - \alpha_{j}}{t_{3} - \alpha_{i} }\right| \nonumber \\ 
  & = &  \log |\lambda_{i,j}| + \sum_{k=2}^{r+s} n_{k}\log  \left|\frac{\lambda_{k}}{\lambda'_{k}}\right|,
  \end{eqnarray*}
where $(i , j)$ is chosen according to Lemma \ref{Tu}, $t_{3} = \frac{x_{3}}{y_{3}}$ and  $n_{k} = n_{k}(x_{3} , y_{3}) \in \mathbb{Z}$.
In order to apply Proposition \ref{mat}, we shall find appropriate values for  the quantities $A_{k}$ and $B$ in the Proposition.
Since Proposition \ref{mat} gives a better lower bound for linear forms in fewer number of logarithms, we will assume that $\lambda_{i,j}$ and $\frac{\lambda_{k}}{\lambda'_{k}}$ are multiplicatively independent and $T_{i,j}(x_{5} , y_{5}) $ is a linear form in $r+s$ logarithms. Recall  that $r+s \leq r+2s = n$.

Let $\lambda$ be a unit in the number field $\mathbb{Q}(\alpha_{i})$ and $\lambda'$ be its corresponding algebraic conjugate in $\mathbb{Q}(\alpha_{j})$. Let $d$ be the degree of $\mathbb{Q}(\alpha_{i}, \alpha_{j})$ over $Q$.
Then  $\lambda/\lambda'$ is a unit in $\mathbb{Q}(\alpha_{i}, \alpha_{j})$ and 
\begin{eqnarray*}
d h\left(\frac{\lambda}{\lambda'}\right) & = & \frac{1}{2} \left|\log\left(\tau\left(\frac{\lambda}{\lambda'}\right)\right)\right|_{1}\\
&=& \frac{1}{2}\left|\log\left(\tau(\lambda)\right)\right|_{1} + \frac{1}{2}\left|\log\left(\tau(\lambda')\right)\right|_{1}\\
& = & nh(\lambda) + nh(\lambda').
\end{eqnarray*}
We also have
$$
h(\lambda') = h(\lambda)  = \frac{1}{2n}\left|\log\left(\tau(\lambda)\right)\right|_{1}. 
$$
Here $| \  |_{1}$ is the $L_{1}$ norm on $\mathbb{R}^{s+t-1}$ and mappings $\tau$ and $\log$ are defined in (\ref{ta}) and (\ref{lo}).  So we have
$$
h(\lambda) = \frac{1}{2n}\left|\log\left(\tau(\lambda)\right)\right|_{1} \leq \frac{\sqrt{2}}{2n} \left\|\log\left(\tau(\lambda)\right)\right\| ,
$$
where $\|. \|$ is the $L_{2}$ norm on $\mathbb{R}^{r+s-1}$ . 
So when $\lambda$ is a unit
\begin{equation}\label{h1}
\max \{dh(\frac{\lambda}{\lambda'}) , \left|\log(\frac{\lambda}{\lambda'})\right| \} \leq \sqrt{2} \left\|\log\left(\tau(\lambda)\right)\right\|.
\end{equation}
Therefore, by Corollary \ref{m} we may choose the values $A_{k}$ so that
 $$
  A_{k} \leq 2\sqrt{2} r_{1},  \   \,    \textrm{for} \ \    2\leq k \leq r+s .
 $$

Let $d_{1}$ be the degree of $\mathbb{Q}(\alpha_{i}, \alpha_{j}, \alpha_{n})$ over $\mathbb{Q}$. Then $d_{1} \leq n (n-1) (n-2)$. We shall find a value for
$A_{1} $ that is at least $ \max \{dh(\gamma_{1}) , |\log \gamma_{1}| \}$ (see the statement of Proposition \ref{mat}).
The following Lemma allows us to take
$$
\frac{A_{1}}{d_{1}} = 2\log 2 + \frac{4}{\sqrt{n}}r_{1}.
$$

\begin{lemma}\label{EstimateOfHeightDelta}
Let $F$ be a binary form of degree $n$ at least $3$ and with integral coefficients.
Assume  $(x,y)$ is a solution to (\ref{1.2}) with $y > M(F)^{1+(n -1)^2}$.
Then we have
\begin{equation*}
	h\left(\frac{\alpha_k-\alpha_i}{\alpha_k-\alpha_j}\right)
	\le
	2\log 2 + \frac{4}{\sqrt{n}}\|\phi(x,y)\|.
\end{equation*}
\end{lemma}
\begin{proof}
Let, $\beta_i = x-y \alpha_i$.
We have
\begin{equation*}
	\frac{\alpha_k-\alpha_i}{\alpha_k-\alpha_j}
	=
	\frac{\beta_k-\beta_i}{\beta_k-\beta_j}.
\end{equation*}
Thus, Lemma \ref{p31} implies that 
\begin{equation}\label{ok26}
	h\left(\frac{\alpha_k-\alpha_i}{\alpha_k-\alpha_j}\right)
	\le 2\log 2 + 4h(\beta_k).
\end{equation}
Set $v_i = \log|\beta_i| = \phi_i(x,y)-\phi_i(1,0)$ for $i=1,2,\ldots n$ and
$\vec{v}=(v_1,v_2,\ldots,v_n)$.
Since $\beta_k$ is a unit, we have
\begin{equation*}
	h(\beta_k) = \frac{1}{2n}\sum_{i=1}^n \left|v_i\right|
		= \frac{1}{2n} (s_1, s_2, \ldots, s_n) \cdot \vec{v}
\end{equation*}
for some  $s_1, s_2, \ldots, s_n\in\{+1, -1\}$.
Noting that $\|(s_1, s_2, \ldots, s_n)\| = \sqrt{n}$, we get
\begin{equation*}
	h(\beta_k) \le \frac{1}{2\sqrt{n}}\|\vec{v}\|.
\end{equation*}
On the other hand, by Lemma \ref{lem100} we have
\begin{equation*}
	\|\vec{v}\| \le \|\phi(x,y)\| + \|\phi(1,0)\|
		\le 2\|\phi(x,y)\|.
\end{equation*}
This, together with (\ref{ok26}), completes the proof.
\end{proof}

Put 
$$
B = \max \left\{1 , \max\{b_{k}A_{k}/A_{1}: \  1\leq k  \leq r+s\}\right\} .
$$
To estimate $B$, we note that since we have chosen $\tau(\lambda_{k})$ ($ 2 \leq k\leq r+s$) so that they are successive minima
 for the lattice $\Lambda$ (see Section \ref{LC}), we have
$$
m_{k}\left\| \log \tau(\lambda_{k}) \right\| \leq  \left\|  \phi(x_{3} , y_{3}) \right\| +  \left\|  \phi(1 , 0) \right\| < 2 \left\|  \phi(x_{3} , y_{3}) \right\|.
$$
 Hence,  we may take
$B \leq r_{3}$, since $A_{1} > 2$. We estimate other values of the quantities in Proposition \ref{mat} as follows:
\begin{eqnarray*}
d &\leq& n!,\\
C_{n} &\leq& \frac{60\, \exp(n) (n+1)^{n+1} 2^{2n+2}(n+2)(n+5/2)n^2}{n!} ,\\
C_{0} &\leq& 4 \log n!,\\
W_{0} &\leq& 2 \log r_{3}.
\end{eqnarray*}
Proposition \ref{mat} implies that 
\begin{eqnarray*}
\log T_{i, j} (x_{3} , y_{3}) & > &- K \log r_{3} r_{1}^{r+s}\\
& >& - K \log r_{3} r_{1}^{n},
\end{eqnarray*}
where the constant $K$ can be taken equal to
\begin{equation}\label{vk}
480\, \exp(n) (n+1)^{n+1} 2^{7n+3/2}(n+2)(n+5/2)n^{5/2} (n-1)(n-2) n! \log (n!).
\end{equation}
Comparing this with Lemma \ref{Tu}, we have
$$
-\log \left(M(F)^{n(n-1)}\right) +\log\left( \sqrt{ \frac{2}{n-2}}\right)+ \frac{-4r_{3}}{(n+1)^2}> -K\,\log r_{3} r_{1}^{n-1},
$$
 By Lemma \ref{Dr} and since  $|D|> D_{0}(n)$, the value $r_{3}$ is large enough to satisfy
$$
 r_{3}^{\frac{e-1}{e}} <  \frac{r_{3}}{\log r_{3}}.
$$
So we may find a constant 
$K_{1}$ depending only on $n$
(see the values of $C(n)$, $C_{0}$ and $W_{0}$ in Proposition \ref{mat}) so that
$$
r_{3} < K_{1}\,r_{1}^{\frac{e}{e-1}n}.
$$
Notice that $K_{1}$ may be chosen equal to
$$
\left(\frac{(n+1)^2}{4} K\right)^{\frac{e}{e-1}}
$$
By Lemma \ref{exg}, we have  
\begin{eqnarray*}
& & M(F)^{n (n -1)} \exp\left(\frac{4r_{1}}{(n+1)^2}\right)\frac{\sqrt{3}}{256}  \, \left(\frac{\log \log n}{\log n}\right)^6  \\
& < & K_{1}\,r_{1}^{1.6n}
\end{eqnarray*}
This is a contradiction,  
as in the above inequality the left hand side is  greater than the right hand side. Hence, related to a root of $F(x , 1) = 0$,  there are at most $2$ solutions $(x , y) \not \in \mathfrak{A}$, with 
$y > M(F)^{1+(n-1)^2}$. 
To see the contradiction, one can consider two different cases. If  $\frac{4r_{1}}{(n+1)^2} > {\frac{e}{e-1}n}^{\frac{e}{e-1}}$ then 
$\exp\left(\frac{4r_{1}}{(n+1)^2}\right) > \frac{4r_{1}}{(n+1)^2}$ and 
 by (\ref{mahD5}) and since $|D| \geq 2^{22} (n+1)^{10} n^n $, the value $M(F)^{n (n -1)}\frac{\sqrt{3}}{256}  \, \left(\frac{\log \log n}{\log n}\right)^6$ exceeds the rest of right hand side.  If  $\frac{4r_{1}}{(n+1)^2} \leq {\frac{e}{e-1}n}^{\frac{e}{e-1}}$ then the value $M(F)^{n (n -1)}\frac{\sqrt{3}}{256}  \, \left(\frac{\log \log n}{\log n}\right)^6$ alone exceeds the right hand side.

\textbf{Remark.} To estimate the value of $A_{1}$ we proved Lemma \ref{EstimateOfHeightDelta}. Having the inequality 
\begin{equation*}
	h\left(\frac{\alpha_n-\alpha_i}{\alpha_n-\alpha_j}\right)
	\le
	2\log 2 + 4h(\alpha_n)
\end{equation*}
in hand, one may attempt to bound the logarithmic height of $\alpha$, a root of $F(x , 1) =0$, in terms of the discriminant of $F$. To do so recall that 
we have assumed that the binary form $F$ has the smallest Mahler measure among all equivalent forms that are monic. We need this assumption to obtain an upper bound for the number of small solutions (see (\ref{S60})). We also have
$$
h(\alpha) = \frac{1}{n} \log M(\alpha) \leq \frac{1}{n} \log\left( (n+1)^{1/2} H(\alpha)\right).
$$
Therefore, we can apply Proposition \ref{gyor5} to our selected form $F$ and assume that for each root $\alpha$ of $F(x , 1)= 0$, we have  
$$
h(\alpha) \leq \frac{1}{n} \log\left( (n+1)^{1/2}\exp\{ n^{4n^{12}} |D|^{6n^8}\}\right).
$$
This will provide an explicit value for $A_{1}$. Should one wish to use this to establish a contradiction similar as above, one has to start  with $5$ solutions (instead of $3$) and after the contradiction, concludes that there are at most $4$ solutions (instead of $2$)   with large $y$  related to each root.

\section{Acknowledgments}
Part of this work has been done while I was supported by Hausdorff Institute for Mathematics in Bonn. 
I would like to thank 
 Professor Yann Bugeaud, Professor Jan-Hendrik Evertse, Professor Andrew Granville, Professor K\'{a}lm\'{a}n Gy\H{o}ry 
and Professor Ryotaro Okazaki 
for their helpful suggestions  and comments.  The content of this manuscript is improved due to   the referee's  care over the details and presentation.


\end{document}